\documentclass{article}
\usepackage{amsfonts}
\usepackage{amsmath}
\usepackage{graphicx, verbatim}
\usepackage{tikz}

\usetikzlibrary{decorations.pathreplacing}
\usetikzlibrary{snakes}
\newtheorem{theorem}{Theorem}

\newtheorem{lemma}[theorem]{Lemma}

\newenvironment{proof}[1][Proof]{\noindent\textbf{#1.} }{\ \rule{0.5em}{0.5em}}

\begin{document}

\title{Subsequential tightness of the maximum of two dimensional
Ginzburg-Landau fields}
\author{ Wei Wu\textsuperscript{1}, Ofer Zeitouni\textsuperscript{2}}
\date{February 26, 2018}
\maketitle

\begin{abstract}
We prove the subsequential tightness of centered maxima of two-dimensional
Ginzburg-Landau fields with bounded elliptic contrast.
\end{abstract}

\footnotetext[1]{%
Statistics department, University of Warwick, Coventry CV4 7AL, UK. E-mail:
W.Wu.9@warwick.ac.uk} \footnotetext[2]{%
Department of Mathematics, Weizmann Institute of Science, Rehovot 7610001,
Israel, and Courant Institute, New York University, 251 Mercer St., New
York, New York 10012. E-mail: ofer.zeitouni@weizmann.ac.il. Supported in
part by the ERC Advanced grant LogCorrelatedFields.}

\section{Introduction}

Let $V\in C^{2}\left( \mathbb{R}\right) $ satisfy 
\begin{gather}
\text{ }V\left( x\right) =V\left( -x\right) ,  \label{a1} \\
\text{ }0<c_{-}\leq V^{\prime \prime }\left( x\right) \leq c_{+}<\infty 
\text{,}  \label{a2}
\end{gather}%
{where }$c_{-},c_{+}${\ are positive constants.} The ratio $\kappa=c_+/c_-$
is called the \textit{elliptic contrast} of $V$. We assume \eqref{a1} and %
\eqref{a2} throughout this note without further mentioning it.

We treat $V$ as a a nearest neighbor potential for a two dimensional
Ginzburg-Landau gradient field. Explicitly, let $D_{N}:=\left[ -N,N\right]
^{2}\cap \mathbb{Z}^{2}$ and let the boundary $\partial D_{N}$ consist of
the vertices in $D_{N}$ that are connected to $\mathbb{Z}^{2}\setminus D_{N}$
by some edge. The Ginzburg-Landau field on $D_{N}$ with zero boundary
condition is a random field denoted by $\phi ^{D_{N},0}$, whose distribution
is given by the Gibbs measure 
\begin{equation}
d\mu_N =Z_{N}^{-1}\exp \left[ -\sum_{v\in D_{N}}\sum_{i=1}^{2}V\left( \nabla
_{i}\phi \left( v\right) \right) \right] \prod_{v\in D_{N}\backslash
\partial D_{N}}d\phi \left( v\right) \prod_{v\in \partial D_{N}}\delta
_{0}\left( \phi \left( v\right) \right) ,  \label{GL}
\end{equation}%
where $\nabla _{i}\phi \left( v\right) =\phi \left( v+e_{i}\right) -\phi
\left( v\right) $, {$e_{1}=(1,0)$ and $e_{2}=(0,1)$}, and $Z_{N}$ is the
normalizing constant ensuring that $\mu_N$ is a probability measure, i.e. $%
\mu_N(\mathbb{R}^{|D_N|})=1$. We denote expectation with respect to $\mu_N$
by $\mathbb{E}_N$, or simply by $\mathbb{E}$ when no confusion can occur.

Ginzburg-Landau fields with convex potential, which are natural
generalizations of the standard lattice Gaussian free field corresponding to
quadratic $V$ (DGFF), have been extensively studied since the seminal works 
\cite{FS,HS,NS}. Of particular relevance to this paper is Miller's coupling,
described in Section \ref{sec:hc} below, which shows that certain
multi-scale decompositions that hold for the Gaussian case continue to hold,
approximately, for the Ginzburg-Landau model.

In this paper, we study the maximum of Ginzburg-Landau fields. Given $%
U\subset D_{N}$, let 
\begin{equation*}
M_{U}:=\max_{x\in U}\phi ^{D_{N},0}\left( x\right) ,
\end{equation*}%
and set $M_{N}=M_{D_{N}}$. For the Gaussian case, we write $M_{N}^{G}$ for $%
M_{N}$. Much is known about $M_{N}^{G}$, following a long succession of
papers starting with \cite{Br83}. In particular, see \cite{BDZ} and \cite%
{BL2}, $M_{N}^{G}-m_{N}^{G}$ converges in distribution to a randomly shifted
Gumbel, with $m_{N}^{G}=c_{1}\log N-c_{2}\log \log N$ and explicit constants 
$c_{1},c_{2}$.

Much less is known concerning the extrema in the Ginzburg-Landau setup, even
though linear statistics of such fields converge to their Gaussian
counterparts \cite{NS}. A first step toward the study of the maximum was
undertaken in \cite{BW}, where the following law of large numbers is proved: 
%result for $M_{D_{N}}$ is
%proved,%
\begin{equation}
\frac{M_{D_{N}}}{\log N}\rightarrow 2\sqrt{g}\text{ in }L^{2}\text{, \ for
some }g=g\left( c_{+},c_{-}\right) .  \label{LLN}
\end{equation}

In this note we prove that the fluctuations of $M_{D_{N}}$ around its mean
are tight, at least along some (deterministic) subsequence.

\begin{theorem}
\label{main}There is a deterministic sequence $\left\{ n_{k}\right\} $ with $%
n_k\to_{k\to\infty} \infty$ such that the sequence of random variables $%
\left\{ M_{D_{n_{k}}}-\mathbb{E}M_{D_{n_{k}}}\right\} $ is tight.
\end{theorem}

As will be clear from the proof, the sequence $\left\{ n_{k}\right\} $ can
be chosen with density arbitrarily close to $1$. Theorem \ref{main} is the
counterpart of an analogous result for the Gaussian case proved in \cite%
{BDeZ}, building on a technique introduced by Dekking and Host \cite{DH91}.
The Dekking-Host technique is also instrumental in the proof of Theorem \ref%
{main}. However, due to the fact that the Ginzburg-Landau field does not
possess good decoupling properties near the boundary, significant changes
need to be made. Additional crucial ingredients in the proof are Miller's
coupling and a decomposition in differences of harmonic functions introduced
in \cite{BW}.

\section{Preliminaries}

\subsection{The Brascamp-Lieb inequality}

One can bound the variances and exponential moments with respect to the
Ginzburg-Landau measure by those with respect to the Gaussian measure, using
the following Brascamp-Lieb inequality. Let $\phi $ be sampled from the
Gibbs measure (\ref{GL}). Given $\eta \in \mathbb{R}^{D_{N}}$, set 
\begin{equation*}
\left\langle \phi ,\eta \right\rangle :=\sum_{v\in D_{N}}\phi _{v}\eta
\left( v\right) .
\end{equation*}

\begin{lemma}[Brascamp-Lieb inequalities \protect\cite{BL}]
Assume that $V\in C^{2}\left( \mathbb{R}\right) $ satisfies $\inf_{x\in 
\mathbb{R}}V^{\prime \prime }\left( x\right) \geq c_{-}>0$. Let $\mathbb{E}_{%
\textup{GFF}}$ and $\textup{Var}_\textup{GFF}$ denote the expectation and
variance with respect to the \textup{DGFF}\ measure (that is, (\ref{GL})
with $V( x) =x^{2}/2$). Then for any $\eta \in \mathbb{R}^{D_{N}}$, 
\begin{eqnarray}
\textup{Var} \langle \phi ,\eta \rangle \quad & \leq & c_{-}^{-1}\textup{%
Var}_\textup{GFF}\left\langle \phi ,\eta \right\rangle ,
\label{eq: BL var bound} \\
\mathbb{E}\left[ \exp \left( \left\langle \phi ,\eta \right\rangle -\mathbb{E%
}\left\langle \phi ,\eta \right\rangle \right) \right] \quad \quad & \leq &
\exp \left( \frac{1}{2}c_{-}^{-1} \textup{Var}_\textup{GFF} \left\langle
\phi ,\eta \right\rangle \right) .  \label{eq: BL exp bound}
\end{eqnarray}
\end{lemma}

\subsection{Approximate harmonic coupling\label{sec:hc}}

By their definition, the Ginzburg-Landau measures satisfy the domain Markov
property: conditioned on the values on the boundary of a domain, the field
inside the domain is again a gradient field with boundary condition given by
the conditioned values. For the discrete GFF, there is in addition a nice
orthogonal decomposition. More precisely, the conditioned field inside the
domain is the discrete harmonic extension of the boundary value to the whole
domain plus an \textit{independent }copy of a \textit{zero boundary}
discrete GFF.

While this exact decomposition does not carry over to general
Ginzburg-Landau measures, the next result due to Jason Miller, see \cite{M},
provides an approximate version.

\begin{theorem}[\protect\cite{M}]
\label{decouple} Let $D\subset \mathbb{Z}^{2}$ be a simply connected domain
of diameter $R$, and denote $D^{r}=\left\{ x\in D:\textup{dist}(x,\partial
D)>r\right\} $. Let $\Lambda $ be such that $f:\partial D\rightarrow \mathbb{%
R}$ satisfies $\max_{x\in \partial D}\left\vert f\left( x\right) \right\vert
\leq \Lambda \left\vert \log R\right\vert ^{\Lambda }$. Let $\phi $ be
sampled from the Ginzburg-Landau measure (\ref{GL}) on $D$ with zero
boundary condition, and $\phi ^{f}$ be sampled from Ginzburg-Landau measure
on $D$ with boundary condition $f$. Then there exist constants $c,\gamma
,\delta ^{\prime }\in \left( 0,1\right) $, that only depend on $V$, so that
if $r>cR^{\gamma }$ then the following holds. There exists a coupling $%
\left( \phi ,\phi ^{f}\right) $, such that if $\hat{\phi}:D^{r}\rightarrow 
\mathbb{R}$ is discrete harmonic with $\hat{\phi}|_{\partial D^{r}}=\phi
^{f}-\phi |_{\partial D^{r}}$, then 
\begin{equation*}
\mathbb{P}\left( \phi ^{f}=\phi +\hat{\phi}\text{ in }D^{r}\right) \geq
1-c\left( \Lambda \right) R^{-\delta ^{\prime }}.
\end{equation*}
\end{theorem}

Here and in the sequel of the paper, for a set $A\subset \mathbb{Z}^2$ and a
point $x\in \mathbb{Z}^2$, we use $\textup{dist}(x,A)$ to denote the
(lattice) distance from $x$ to $A$.

\subsection{Pointwise tail bound}

We also recall the pointwise tail bound for the Ginzburg-Landau field (\ref%
{GL}), proved in \cite{BW}.

\begin{theorem}
\label{thm: tail bound} Let $g$ be the constant as in (\ref{LLN}). For all $%
u>0$ large enough and all $v\in D_{N}$ we have

\begin{equation}
\mathbb{P}\left( \phi _{v}\geq u\right) \leq \exp \left( -\frac{u^{2}}{%
2g\log \textup{dist}( v,\partial D_{N}) }+o( u) \right) .  \label{inn}
\end{equation}
\end{theorem}

This allows us to conclude that the maximum of $\phi ^{D_{N},0}$ does not
occur within a thin layer near the boundary.

\begin{lemma}
\label{thin} Given $\delta <1$, there exists $\delta ^{\prime }>0$ such that%
\begin{equation*}
\mathbb{P}\left( M_{A_{N,N^{\delta }}}>\left( 2\sqrt{g}-\delta ^{\prime
}\right) \log N\right) \leq N^{\frac{\delta -1}{2}},
\end{equation*}%
where%
\begin{equation*}
A_{N,N^{\delta }}:=\left\{ v\in D_{N}:\textup{dist}( x,\partial D_{N})
<N^{\delta }\right\} .
\end{equation*}
\end{lemma}

\begin{proof}
Let $\Delta =\textup{dist}( x,\partial D_{N}) $. For $\delta ^{\prime }$
small enough, applying Theorem \ref{thm: tail bound} with $u=\left( 2\sqrt{g}%
-\delta ^{\prime }\right) \log N$ yields%
\begin{eqnarray*}
P\left( \phi _{v}\geq \left( 2\sqrt{g}-\delta ^{\prime }\right) \log
N\right) &\leq &\exp \left( -2\frac{\left( \log N\right) ^{2}}{\log \Delta }+%
\frac{2\delta ^{\prime }}{\sqrt{g}}\frac{\left( \log N\right) ^{2}}{\log
\Delta }+o\left( \log N\right) \right) \\
&\leq &N^{-2+2\delta ^{\prime }/\sqrt{g}+o\left( 1\right) },\text{ \ \ for
all }v\in A_{N,N^{\delta }}.
\end{eqnarray*}%
Therefore a union bound yields 
\begin{equation*}
P\left( M_{A_{N,N^{\delta }}}\geq \left( 2\sqrt{g}-\delta ^{\prime }\right)
\log N\right) \leq N^{\delta -1+2\delta ^{\prime }/\sqrt{g}+o\left( 1\right)
}.
\end{equation*}%
It suffices to take $\delta ^{\prime }$ such that $2\delta ^{\prime }/\sqrt{g%
}<\frac{1-\delta }{2}$.
\end{proof}

\section{The recursion and proof of Theorem \protect\ref{main}}

We prove Theorem \ref{main} by establishing a recursion for some random
variable $M_{Y_{N}}$, where $Y_{N}\subset D_{N}$ is a specific subset
defined below. Denote by $T_N=[-N,N]\times \{N\}\subset D_N$ the top
boundary of $D_{N}$. For fixed $\varepsilon >0$, define%
\begin{equation*}
Y_{N}=\left\{ v\in D_{N}:\textup{dist}( v,\partial D_{N}) \geq \varepsilon
N\right\} \cup \left\{ v\in D_{N}:\textup{dist}( v,\partial D_{N}) =%
\textup{dist}(v,T_N) \right\}.
\end{equation*}%
For $\delta \in \left( 0,1\right) $, we also define $Y_{N,\delta }\subset
Y_{N}$ as 
\begin{equation*}
Y_{N,\delta }=\left\{ v\in Y_{N}:\textup{dist}( v,T_N) >N^{1-\delta
}\right\} ,
\end{equation*}%
see Figure \ref{fig:YNd}. 
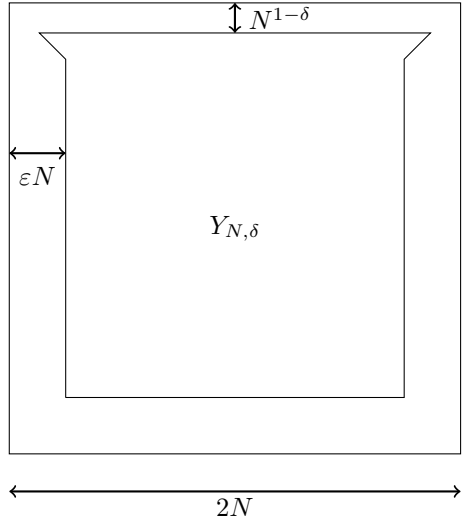
\begin{figure}[h]
\begin{center}
\begin{tikzpicture}
%  [ endstyle/.style={circle,fill=blue,inner sep=1.5pt}
%  ]

\draw (0,0) rectangle (6,6);
\draw (0.75,0.75) -- (5.25,0.75) -- (5.25,5.25) -- (5.6,5.6) -- (0.4,5.6) -- (0.75,5.25) -- (0.75,0.75);
\draw[thick,<->] (0,-0.5) -- (6,-0.5);
\draw[thick,<->] (0,4) -- (0.75,4);
\draw[thick,<->] (3,6) -- (3,5.6);
\node  at (0.375, 3.7) {$\varepsilon N$};
\node at (3.6,5.8) {$N^{1-\delta}$};
\node  at (3,-0.7) {$2N$};
\node at (3,3) {$Y_{N,\delta}$};
\end{tikzpicture}
\end{center}
\caption{The domain $Y_{N,\protect\delta}$.}
\label{fig:YNd}
\end{figure}
%(ADD FIG 1)

\begin{lemma}
\label{YLLN}For the constant $g=g\left( c_{+},c_{-}\right) $ in (\ref{LLN}),
we have 
\begin{equation}  \label{eq-YLLN}
\frac{M_{Y_{N,\delta }}}{\log N}\rightarrow 2\sqrt{g}\text{ in }L^{2}.
\end{equation}
\end{lemma}

\begin{proof}
Let $D_{N}^{\varepsilon }:=\left\{ v\in D_{N}:\textup{dist}(v,\partial
D_{N}) \geq \varepsilon N\right\} $. Since 
\begin{equation*}
\frac{M_{D_{N}^{\varepsilon }}}{\log N}\leq \frac{M_{Y_{N,\delta }}}{\log N}%
\leq \frac{M_{D_{N}}}{\log N},
\end{equation*}
the claim \eqref{eq-YLLN} follows from \cite{BW}, since the upper control on 
${M_{D_{N}}}/{\log N}$ follows from (\ref{LLN}) while the lower control on ${%
M_{D_{N}^{\varepsilon }}}/{\log N}$ follows from the display below (5.19) in 
\cite{BW}.
\end{proof}

We now switch to dyadic scales. For $n\in \mathbb{N}$, set $N=2^{n}$ and $%
m_{n}:=M_{Y_{2^{n},\delta }}$. We set up a recursion for $m_{n}$. Clearly, 
\begin{equation*}
\mathbb{E}m_{n+2} =\mathbb{E}M_{Y_{4N,\delta }} \geq \mathbb{E}\max \left\{
\max_{v\in Y_{N,\delta }^{\left( 1\right) }}\phi _{v}^{D_{4N},0},\max_{v\in
Y_{N,\delta }^{\left( 2\right) }}\phi _{v}^{D_{4N},0}\right\} ,
\end{equation*}
where $Y_{N,\delta }^{\left( i\right) }$ are the translations of $%
Y_{N,\delta }$, defined by $Y_{N,\delta }^{\left( 1\right) }=Y_{N,\delta
}+\left( -1.1N,3N\right) $, $Y_{N,\delta }^{\left( 2\right) }=Y_{N,\delta
}+\left( 1.1N,3N\right) $, see Figure \ref{fig:YNd1} 
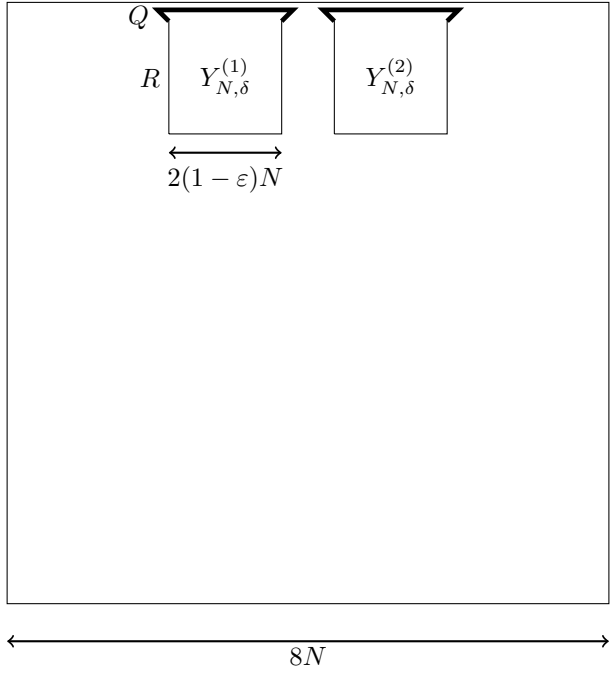
\begin{figure}[h]
\begin{center}
\begin{tikzpicture}
%  [ endstyle/.style={circle,fill=blue,inner sep=1.5pt}
%  ]

\draw (0,0) rectangle (8,8);
\draw (2.15,6.25) -- (3.65,6.25) -- (3.65,7.75) -- (3.8,7.9) -- (2,7.9) -- (2.15,7.75) -- (2.15,6.25);
\draw (4.35,6.25) -- (5.85,6.25) -- (5.85,7.75) -- (6,7.9) -- (4.2,7.9) -- (4.35,7.75) -- (4.35,6.25);
\draw[line width=1.8pt]  (3.65,7.75) -- (3.8,7.9) -- (2,7.9) -- (2.15,7.75);
\draw[line width=1.8pt] (5.85,7.75) -- (6,7.9) -- (4.2,7.9) -- (4.35,7.75);
\node at (2.9,7) {$Y_{N,\delta}^{(1)}$};
\node at (5.1,7) {$Y_{N,\delta}^{(2)}$};
\node at (1.75,7.8) {$Q$};
\node at (1.90,7) {$R$};
\node  at (4,-0.7) {$8N$};
\draw[thick,<->] (0,-0.5) -- (8,-0.5);
\draw[thick,<->] (2.15,6) -- (3.65,6);
\node at (2.9,5.65) {$2(1-\varepsilon)N$};
\end{tikzpicture}
\end{center}
\caption{The domains $Y_{N,\protect\delta}^{(i)}$, with the boundary pieces $%
R,Q$.}
\label{fig:YNd1}
\end{figure}

The next two lemmas will allow us to control the difference between $%
\phi^{D_{4N},0}$ and $\phi ^{D_{N},0}$ (and as a consequence, between $%
m_{n+2}$ and $m_{n}$).

\begin{lemma}
\label{A}There exist $\delta ^{\prime },1>\delta >\gamma >0$, such that the
following statement holds. Set $D_{N}^{\left( 1\right) }=D_{N}+\left(
-1.1N,3N\right) $, $D_{N}^{\left( 2\right) }=D_{N}+\left( 1.1N,3N\right) $.
Let $D_{N}^{\gamma ,\left( i\right) }:=\left\{ v\in D_{N}^{\left( i\right) }:%
\textup{dist}(v,\partial D_{N}^{\left( i\right) }) \geq N^{\gamma }\right\} 
$. Then there exists a coupling $\mathbb{P}$ of \newline
$\left( \phi ^{D_{4N,}0},\phi ^{D_{N}^{\left( 1\right) },0},\phi
^{D_{N}^{\left( 2\right) },0}\right) $ and an event $\mathcal{G}$ with $%
\mathbb{P}\left( \mathcal{G}^{c}\right) \leq N^{-\delta ^{\prime }}$, such
that with $h_{v}^{\left( i\right) }$ being harmonic functions in $%
D_{N}^{\left( i\right) }$ with boundary conditions $\phi ^{D_{4N,}0}-\phi
^{D_{N}^{\left( i\right) },0}$, on the event $\mathcal{G}$, we have%
\begin{equation*}
\phi _{v}^{D_{4N},0}=\phi _{v}^{D_{N}^{\left( i\right) },0}+h_{v}^{\left(
i\right) }\text{, for all }v\in Y_{N,\gamma }^{\left( i\right) },\text{ for }%
i=1,2.
\end{equation*}%
Moreover, there is a constant $C_{0}=C_{0}\left( \delta \right) $, such
that, for any $1>\delta >\gamma $, 
\begin{equation*}
\max_{\substack{ i=1,2  \\ v\in Y_{N,\delta }^{\left( i\right) }}} \textup{%
Var} \left( h_{v}^{\left( i\right) }\right) \leq C_{0}\left( \delta \right) .
\end{equation*}
\end{lemma}

\begin{lemma}
\label{B} With notation as in Lemma \ref{A}, there exists a constant $%
C_{1}<\infty $, such that%
\begin{equation*}
\mathbb{E}\min_{i}\min_{v\in Y_{N,\delta }^{\left( i\right) }}h_{v}^{\left(
i\right) }=-\mathbb{E}\max_{i}\max_{v\in Y_{N,\delta }^{\left( i\right)
}}h_{v}^{\left( i\right) }\geq -C_{1}.
\end{equation*}
\end{lemma}

The proof of Lemmas \ref{A} and \ref{B} are postponed to Section \ref%
{sec-LemAB}. In the rest of this section, we bring the proof of Theorem \ref%
{main}.

\begin{proof}[Proof of Theorem \protect\ref{main}]
Denote by $m_{n}^{\ast }$ an independent copy of $m_{n}$. We combine Lemmas %
\ref{A} and \ref{B} to conclude%
\begin{eqnarray*}
\mathbb{E}m_{n+2} &\geq &\mathbb{E}\left[ 1_{\mathcal{G}}\max_{i}\max_{v\in
Y_{N,\delta }^{\left( i\right) }}\left( \phi _{v}^{D_{N}^{\left( i\right)
},0}+h_{v}^{\left( i\right) }\right) \right] \\
&\geq &\mathbb{E}\max \left\{ m_{n},m_{n}^{\ast }\right\} +2\mathbb{E}%
\min_{i}\min_{v\in Y_{N,\delta }^{\left( i\right) }}h_{v}^{\left( i\right)
}-2\mathbb{E}\left[ 1_{\mathcal{G}^{c}}m_{n}\right] .
\end{eqnarray*}%
We apply (\ref{LLN}) to conclude that 
\begin{equation*}
\mathbb{E}\left[ 1_{\mathcal{G}^{c}}m_{n}\right] \leq \mathbb{P}\left( 
\mathcal{G}^{c}\right) ^{1/2}\mathbb{E}\left[ m_{n}^{2}\right] ^{1/2}\leq C%
\frac{\log N}{N^{\delta ^{\prime }/2}}
\end{equation*}%
Thus for all large $n$, we can apply Lemma \ref{B} to get 
\begin{equation*}
\mathbb{E}m_{n+2}\geq \mathbb{E}\max \left\{ m_{n},m_{n}^{\ast }\right\}
-3C_{1}.
\end{equation*}%
Using $\max \left\{ a,b\right\} =\frac{1}{2}\left( a+b+\left\vert
a-b\right\vert \right) $ and Jensen's inequality, we obtain%
\begin{equation}  \label{eq-ouf}
\mathbb{E}m_{n+2}-\mathbb{E}m_{n} \geq \frac{1}{2}\mathbb{E}\left\vert
m_{n}-m_{n}^{\ast }\right\vert -3C_{1} \geq \frac{1}{2}\mathbb{E}\left\vert
m_{n}-\mathbb{E}m_{n}^{\ast }\right\vert -3C_{1}.
\end{equation}
We need the following lemma.

\begin{lemma}
\label{subseq}There exists a sequence $\left\{ n_{k}\right\} $ and a
constant $K<\infty $ such that%
\begin{equation*}
\mathbb{E}m_{n_{k}+2}\leq \mathbb{E}m_{n_{k}}+K.
\end{equation*}
\end{lemma}

\begin{proof}[Proof of Lemma \protect\ref{subseq}]
Let $I_{n,K}=\left\{ j\in \{n,n+2,\ldots ,2n\}:\mathbb{E}m_{j+2}\geq \mathbb{%
E}m_{j}+K\right\} $. Then Lemma \ref{YLLN} implies that $\mathbb{E}%
m_{n}/n\rightarrow 2\sqrt{g}$, while \eqref{eq-ouf} gives $\mathbb{E}%
m_{n+2}\geq \mathbb{E}m_{n}-3C_{1}$. Therefore, for any fixed $\eta >0$ and
all large $n$, 
\begin{equation*}
K|I_{n,K}|-3C_{1}(\frac{n}{2}-|I_{n,K}|)\leq 2\sqrt{g}n(1+\eta )\leq 4\sqrt{g%
}n,
\end{equation*}%
giving that for all large $n$, 
\begin{equation*}
|I_{n,K}|\leq n\frac{4\sqrt{g}+\frac{3}{2C_{1}}}{K+3C_{1}}.
\end{equation*}%
Choosing $K>16\sqrt{g}+3C_{1}$ gives that for all $n$ large, $|I_{n,K}|\leq
n/4$. %\begin{equation*}
%\limsup_{n\rightarrow \infty }\frac{\left\vert I_{n,K}\right\vert }{2n}\leq 
%\frac{2\sqrt{g}}{K}.
%\end{equation*}%
It follows that for all large $n$, there exists $n^{\prime }\in \left[ n,2n%
\right] $, such that%
\begin{equation*}
\mathbb{E}m_{n^{\prime }+2}\leq \mathbb{E}m_{n^{\prime }}+K,
\end{equation*}%
This completes the proof of Lemma \ref{subseq}.
\end{proof}

We continue with the proof of Theorem \ref{main}. Using the subsequence $%
\left\{ n_{k}\right\} $ from Lemma \ref{subseq}, we have from \eqref{eq-ouf}
that 
\begin{equation*}
\mathbb{E}\left\vert m_{n_{k}}-\mathbb{E}m_{n_{k}}^{\ast }\right\vert \leq
2K+6C_{1},
\end{equation*}%
which implies, using Jensen's inequality, that $\left\{ m_{n_{k}}-\mathbb{E}%
m_{n_{k}}^{\ast }\right\} $ is tight. This implies that the the sequence of
random variables%
\begin{equation*}
\bar M_{D_{N_{k}}^{\delta }}:=\max \left\{ \phi _{v}^{D_{N_{k}},0}:v\in
D_{N_{k}},\textup{dist}(v,\partial D_{N_{k}}) \geq N_{k}^{1-\delta }\right\}
\end{equation*}%
is tight around its mean because $\bar M_{D_{N_{k}}^{\delta }}$ is the
maximum of 4 rotated copies of $m_{n_{k}}$.

Finally, combining (\ref{LLN}) and Lemma \ref{thin} we obtain%
\begin{equation*}
\mathbb{P}\left( M_{D_{N_{k}}}>\bar{M}_{D_{N_{k}}^{\delta }}\right) \leq
2^{n_{k}(\delta -1)/2},
\end{equation*}%
and 
\begin{eqnarray*}
\mathbb{E}M_{D_{N_{k}}}-\mathbb{E}\bar{M}_{D_{N_{k}}^{\delta }} &\leq &%
\mathbb{E}M_{D_{N_{k}}}1_{\left\{ M_{D_{N_{k}}}>\bar{M}_{D_{N_{k}}^{\delta
}}\right\} } \\
&\leq &\mathbb{P}\left( M_{D_{N_{k}}}>\bar{M}_{D_{N_{k}}^{\delta }}\right)
^{1/2}\left( \mathbb{E}M_{D_{N_{k}}}^{2}\right) ^{1/2} \\
&\leq &2^{n_{k}\frac{\delta -1}{4}}O\left( \log N_{k}\right) \rightarrow 0.
\end{eqnarray*}%
We conclude that the sequence $\left\{ M_{D_{N_{k}}}-\mathbb{E}%
M_{D_{N_{k}}}\right\} $ is tight.
\end{proof}

\section{Proof of Lemma \protect\ref{A} and \protect\ref{B}}

\label{sec-LemAB}

\begin{proof}[Proof of Lemma \protect\ref{A}]
The existence of the harmonic decomposition is implied by the Markov
property and Theorem \ref{decouple} (with $\delta ^{\prime },\gamma $ taken
as the constants in Theorem \ref{decouple}). It thus suffices to obtain an
upper bound for Var$\left( h_{v}^{\left( i\right) }\right) $. Write $%
h_{v}^{\left( i\right) }=\hat{h}_{v}^{\left( i\right) }-\tilde{h}%
_{v}^{\left( i\right) }$, where $\hat{h}_{v}^{\left( i\right) }$ is the
harmonic function in $D_{N}^{\gamma ,\left( i\right) }$ with boundary value $%
\phi ^{D_{4N},0}$, and $\tilde{h}_{v}^{\left( i\right) }$ is the harmonic
function in $D_{N}^{\gamma ,\left( i\right) }$ with boundary value $\phi
^{D_{N}^{\left( i\right) },0}$. Without loss of generality we set $i=1$.
Applying the Brascamp-Lieb inequality (\ref{eq: BL var bound}) we get%
\begin{equation*}
\textup{Var} \left( h_{v}^{\left( 1\right) }\right) \leq c_{-}^{-1}\textup{%
Var}_\textup{GFF} \left( h_{v}^{\left( 1\right) }\right) .
\end{equation*}%
The orthogonal decomposition for GFF implies%
\begin{eqnarray*}
\textup{Var}_\textup{GFF} \left( \hat{h}_{v}^{\left( 1\right) }\right) &=&%
\textup{Var}_\textup{GFF} \left( \mathbb{E}_{\textup{GFF}}\left[ \phi
_{v}^{D_{4N},0}|\mathcal{F}_{\partial D_{N}^{\gamma ,\left( 1\right) }}%
\right] \right) \\
&=&\textup{Var}_\textup{GFF} \left[ \phi _{v}^{D_{4N},0}\right] -\textup{%
Var}_\textup{GFF} \left[ \phi _{v}^{D_{N}^{\gamma ,\left( 1\right) },0}%
\right]
\end{eqnarray*}%
and%
\begin{equation*}
\textup{Var}_\textup{GFF} \left( \tilde{h}_{v}^{\left( 1\right) }\right) =%
\textup{Var}_\textup{GFF} \left[ \phi _{v}^{D_{N}^{\left( 1\right) },0}%
\right] -\textup{Var}_\textup{GFF} \left[ \phi _{v}^{D_{N}^{\gamma ,\left(
1\right) },0}\right] .
\end{equation*}%
We now estimate the last two expressions for different regions of $v\in
Y_{N,\delta }^{\left( 1\right) }$. First of all, it suffices to control $%
h_{v}^{\left( 1\right) }$ for $v\in \partial Y_{N,\delta }^{\left( 1\right)
} $. Let 
\begin{eqnarray*}
Q &:&=\left\{ v\in \partial Y_{N,\delta }^{\left( 1\right) }:\textup{dist}%
(v,\partial D_{N}) =\textup{dist}(v,T) \right\} , \\
R &:&=\left\{ v\in \partial Y_{N,\delta }^{\left( 1\right) }:\textup{dist}%
(v,\partial D_{N}) \text{ =}\varepsilon N\right\} .
\end{eqnarray*}%
We first show that%
\begin{eqnarray}
\max_{v\in R}\textup{Var}_\textup{GFF} \left( \hat{h}_{v}^{\left( 1\right)
}\right) &\leq &C\left( \varepsilon \right) ,  \notag \\
\max_{v\in Q\cup R}\textup{Var}_\textup{GFF} \left( \tilde{h}_{v}^{\left(
1\right) }\right) &\leq &C_{0}N^{\gamma -\delta }.  \label{vartilt}
\end{eqnarray}%
Indeed, standard asymptotics for the lattice Green's function (following
e.g. from \cite[Proposition 1.6.3]{Law13}) give, for some constant $g_{0}$, 
\begin{eqnarray*}
&&\!\!\!\!\!\!\textup{Var}_\textup{GFF} \left[ \phi _{v}^{D_{4N},0}\right]
-\textup{Var}_\textup{GFF} \left[ \phi _{v}^{D_{N}^{\gamma ,\left(
1\right) },0}\right] \\
&=&g_{0}\left( \log \textup{dist}(v,\partial D_{4N}) -\log \textup{dist}%
(v,\partial D_{N}^{\gamma ,\left( 1\right) }) \right) +o_{N}\left( 1\right)
\\
&\leq &g_{0}\log \frac{4N}{\varepsilon N-N^{\gamma }}+o_{N}\left( 1\right)
\leq C(\varepsilon ),
\end{eqnarray*}%
and similarly,%
\begin{eqnarray*}
&&\!\!\!\!\!\!\textup{Var}_\textup{GFF} \left[ \phi _{v}^{D_{N}^{\left(
1\right) },0}\right] -\textup{Var}_\textup{GFF} \left[ \phi
_{v}^{D_{N}^{\gamma ,\left( 1\right) },0}\right] \\
&=&g_{0}\left( \log \textup{dist}(v,\partial D_{N^{{}}}^{\left( 1\right) })
-\log \textup{dist}(v,\partial D_{N}^{\gamma ,\left( 1\right) }) \right)
+O\left( N^{-1}\right) \\
&\leq &g_{0}\log \frac{N^{\delta }}{N^{\delta }-N^{\gamma }}+O\left(
N^{-1}\right) \leq C_{0}N^{\gamma -\delta }.
\end{eqnarray*}

To conclude the proof, we also claim for $\delta \in \left( \gamma ,1\right) 
$%
\begin{equation}
\max_{v\in Q}\textup{Var}_\textup{GFF} \left( \hat{h}_{v}^{\left( 1\right)
}\right) \leq CN^{\gamma -\delta }.  \label{varq}
\end{equation}%
Indeed, denote by $T_{\gamma }$ the top boundary of $D_{N}^{\gamma }$, we
apply asymptotics for lattice Green's function to obtain%
\begin{eqnarray*}
&&\!\!\!\!\!\!\textup{Var}_\textup{GFF} \left[ \phi _{v}^{D_{4N},0}\right]
-\textup{Var}_\textup{GFF} \left[ \phi _{v}^{D_{N}^{\gamma },0}\right] \\
& &=g_{0}\left( \log \textup{dist}(v,\partial D_{4N}) -\log \textup{dist}%
(v,\partial D_{N}^{\gamma }\right) ) +O\left( N^{-1}\right) \\
&&=g_{0}\left( \log \textup{dist}(v,T) -\log \textup{dist}( v,T_{\gamma })
\right) +O\left( N^{-1}\right) .
\end{eqnarray*}%
Since 
\begin{equation*}
\log \frac{\textup{dist}(v,T) }{\textup{dist}(v,T_{\gamma }) }\leq \log 
\frac{N^{\delta }}{N^{\delta }-N^{\gamma }}\leq CN^{\gamma -\delta },
\end{equation*}%
we obtain (\ref{varq}).
\end{proof}

\begin{proof}[Proof of Lemma \protect\ref{B}]
Recall that $h_v^{(i)}=\hat{h}_v^{(i)}-\tilde h_{v}^{(i)}$. We will prove
that there exist $C_{0}<\infty $ and $\alpha >0$, such that for all $%
C_{1}>C_{0}$, 
\begin{eqnarray}
\mathbb{P}\left( \max_{v\in Q}\hat{h}_{v}^{\left( 1\right) }>C_{1}\right)
&\leq &e^{-\alpha C_{1}},  \label{qtail} \\
\mathbb{P}\left( \max_{v\in R}\hat{h}_{v}^{\left( 1\right) }>C_{1}\right)
&\leq &e^{-\alpha C_{1}},  \label{rtail} \\
\mathbb{P}\left( \min_{v\in Q\cup R}\tilde{h}_{v}^{\left( 1\right)
}<-C_{1}\right) &\leq &e^{-\alpha C_{1}}.  \label{tilth}
\end{eqnarray}%
Indeed, (\ref{qtail}) follows from (\ref{varq}) and the exponential
Brascamp-Lieb inequality (\ref{eq: BL exp bound}):%
\begin{eqnarray*}
\mathbb{P}\left( \max_{v\in Q}\hat{h}_{v}^{\left( 1\right) }>C_{1}\right)
&\leq &\left\vert Q\right\vert \max_{v\in Q} \mathbb{P}\left( \hat{h}%
_{v}^{\left( 1\right) }>C_{1}\right) \\
&\leq &C_{3}N\exp \left( -\frac{C_{1}^{2}}{C_{2}\textup{Var}_\textup{GFF}
\left( \hat{h}_{v}^{\left( 1\right) }\right) }\right) \\
&\leq &C_{3}N\exp \left( -\frac{C_{1}^{2}}{C_{2}}N^{\delta -\gamma }\right) ,
\end{eqnarray*}%
where $C_2,C_3$ are some fixed constants. The same argument using (\ref%
{vartilt}) gives (\ref{tilth}).

We now prove (\ref{rtail}) using chaining. Omitting the superscripts $\left(
1\right) $ in $\hat{h}^{\left( 1\right) }$ and $D_{N}^{\gamma ,\left(
1\right) }$, we claim that there exists $K<\infty $, such that for $u,v\in R$%
, 
\begin{equation}
\textup{Var}_\textup{GFF} \left[ \hat{h}_{u}-\hat{h}_{v}\right] \leq K%
\frac{\left\vert u-v\right\vert }{\varepsilon N}.  \label{chain}
\end{equation}%
Applying the orthogonal decomposition of the DGFF we obtain%
\begin{equation*}
\phi _{u}^{D_{4N},0}-\phi _{v}^{D_{4N},0}=\phi _{u}^{D_{N}^{\gamma },0}-\phi
_{v}^{D_{N}^{\gamma },0}+\hat{h}_{u}-\hat{h}_{v},
\end{equation*}%
and therefore, by the independence of $\phi _{u}^{D_{N}^{\gamma },0}-\phi
_{v}^{D_{N}^{\gamma },0}$ and $\hat{h}_{u}-\hat{h}_{v}$ under the DGFF
measure, 
\begin{equation}
\textup{Var}_\textup{GFF} \left[ \hat{h}_{u}-\hat{h}_{v}\right] =\textup{%
Var}_\textup{GFF} \left[ \phi _{u}^{D_{4N},0}-\phi _{v}^{D_{4N},0}\right] - 
\textup{Var}_\textup{GFF} \left[ \phi _{u}^{D_{N}^{\gamma },0}-\phi
_{v}^{D_{N}^{\gamma },0}\right] .  \label{diff}
\end{equation}%
We now apply the representation of the lattice Green's function, 
%to the right hand
%side above (
see, e.g., \cite[Proposition 1.6.3]{Law13},%
\begin{equation*}
G^{D_{N}}( u,v) =\sum_{y\in \partial D_{N}}H_{\partial D_{N}}( u,y) a( y-v)
-a( u-v) ,
\end{equation*}%
where $H_{\partial D_{N}}\left( u,\cdot \right) $ is the harmonic measure of 
$D_{N}$ seen at $u$ and $a$ is the potential kernel on $\mathbb{Z}^{2}$
which satisfies the asymptotics%
\begin{equation*}
a\left( x\right) =\frac{2}{\pi }\log \left\vert x\right\vert +D_{0}+O\left(
\left\vert x\right\vert ^{-2}\right) ,
\end{equation*}
where $D_0$ is an explicit constant (see e.g. \cite[Page 39]{Law13} for a
slightly weaker result which nevertheless is sufficient for our needs).
Substituting into (\ref{diff}), we see that 
\begin{eqnarray*}
&&\textup{Var}_\textup{GFF} \left[ \phi _{u}^{D_{4N},0}-\phi
_{v}^{D_{4N},0}\right] - \textup{Var}_\textup{GFF} \left[ \phi
_{u}^{D_{N}^{\gamma },0}-\phi _{v}^{D_{N}^{\gamma },0}\right] \\
&=&G^{D_{4N}}\left( u,u\right) +G^{D_{4N}}\left( v,v\right)
-2G^{D_{4N}}\left( u,v\right) \\
&&-\left( G^{D_{N}^{\gamma }}\left( u,u\right) +G^{D_{N}^{\gamma }}\left(
v,v\right) -2G^{D_{N}^{\gamma }}\left( u,v\right) \right) \\
&=&\sum_{z\in \partial D_{4N}}H_{\partial D_{4N}}\left( u,z\right) a\left(
u-z\right) +\sum_{z\in \partial D_{4N}}H_{\partial D_{4N}}\left( v,z\right)
a\left( v-z\right) \\
&& \quad \quad \quad -2\sum_{z\in \partial D_{4N}}H_{\partial D_{4N}}\left(
u,z\right) a\left( v-z\right) \\
&&-\sum_{z\in \partial D_{N}^{\gamma }}H_{\partial D_{N}^{\gamma }}\left(
u,z\right) a\left( u-z\right) -\sum_{z\in \partial D_{N}^{\gamma
}}H_{\partial D_{N}^{\gamma }}\left( v,z\right) a\left( v-z\right) \\
&&\quad \quad \quad +2\sum_{z\in \partial D_{N}^{\gamma }}H_{\partial
D_{N}^{\gamma }}\left( u,z\right) a\left( v-z\right) \\
&:&=A_{D_{4N}}-A_{D_{N}^{\gamma }}
\end{eqnarray*}%
We now apply the Harnack inequality, see \cite[Theorem 1.7.1]{Law13},%
\begin{equation*}
\left\vert H_{\partial D_{4N}}\left( u,z\right) -H_{\partial D_{4N}}\left(
v,z\right) \right\vert \leq \frac{\left\vert u-v\right\vert }{4N}
\end{equation*}%
to obtain%
\begin{eqnarray*}
A_{D_{4N}} &=&\sum_{z\in \partial D_{4N}}H_{\partial D_{4N}}\left(
u,z\right) \left( a\left( u-z\right) -a\left( v-z\right) \right) \\
&&+\sum_{z\in \partial D_{4N}}\left( H_{\partial D_{4N}}\left( v,z\right)
-H_{\partial D_{4N}}\left( u,z\right) \right) a\left( v-z\right) \\
&\leq &\frac{\left\vert u-v\right\vert }{4N}\sum_{z\in \partial
D_{4N}}H_{\partial D_{4N}}\left( u,z\right) \\
&&+\sum_{z\in \partial D_{4N}}\left( H_{\partial D_{4N}}\left( v,z\right)
-H_{\partial D_{4N}}\left( u,z\right) \right) \left( a\left( v-z\right) -%
\frac{2}{\pi }\log N-D_{0}\right) \\
&\leq &K\frac{\left\vert u-v\right\vert }{N},\text{ \ \ \ for some }K<\infty 
\text{.}
\end{eqnarray*}%
The same argument gives $\left\vert A_{D_{N}^{\gamma }}\right\vert \leq K%
\frac{\left\vert u-v\right\vert }{\varepsilon N}$, thus (\ref{chain}) is
proved.

Now fix a large $k_{0}$. For $k\geq k_{0}$ let $P_{k}$ be subsets of $R$
that plays the role of dyadic approximations: $P_{k}$ contains $O\left(
2^{k}\right) $ vertices that are equally spaced and the graph distance
between adjacent points is $\varepsilon N2^{-k}$. For $v\in R$, denote by $%
P_{k}\left( v\right) $ the $k^{th}$ dyadic approximation of $v$, namely the
vertex in $P_{k}$ that is closest to $v$. Then for $v\in R$,%
\begin{equation*}
\hat{h}_{v}=P_{k_{0}}\left( v\right) +\sum_{k\geq k_{0}}\hat{h}%
_{P_{k+1}\left( v\right) }-\hat{h}_{P_{k}\left( v\right) }.
\end{equation*}%
We now apply the exponential Brascamp-Lieb inequality (\ref{eq: BL exp bound}%
), (\ref{chain}), and a union bound to obtain 
\begin{eqnarray*}
&&\!\!\!\!\!\mathbb{P}\left( \max_{v\in R}\left[ \hat{h}_{P_{k+1}\left(
v\right) }-\hat{h}_{P_{k}\left( v\right) }\right] >\sqrt{K\left( \frac{3}{2}%
\right) ^{-k}}\right) \\
&&\leq C_{3}2^{k}\exp \left( -K\left( \frac{3}{2}\right) ^{-k}\frac{C_{4}}{%
2\cdot 2^{-k}}\right) \\
&&\leq C_{3}2^{k}\exp \left( -C_{4}\left( \frac{4}{3}\right) ^{k}\frac{K}{2}%
\right) ,
\end{eqnarray*}%
for some constant $C_4$. Since both $\sqrt{K\left( \frac{3}{2}\right) ^{-k}}$
and the tail probability are summable in $k$, we conclude that (\ref{rtail})
holds.
\end{proof}

\bigskip

\noindent \textbf{Acknowledgment} O.Z. thanks Jason Miller for suggesting,
years ago, that the coupling in \cite{M} could be useful in carrying out the
Dekking-Host argument for the Ginzburg-Landau model. W.W. thanks the
Weizmann Institute for its hospitality. This work was started while both
authors were at the Courant Institute, NYU.

\bibliographystyle{alpha}
\bibliography{GLmax}

\end{document}